\documentclass[a4paper, 11pt]{article}

\usepackage{amsmath,amssymb ,amstext ,mathrsfs, extarrows}
\usepackage[utf8]{inputenc}
\usepackage{float}
\usepackage{dsfont}
\usepackage{yfonts}
\usepackage{amsfonts}
\usepackage{geometry}
\usepackage{stmaryrd}
\usepackage{mathabx}
\usepackage{hyperref}
\usepackage{tikz}
\usetikzlibrary{cd}
\usepackage{xfrac}
\usepackage[arrow, matrix, curve]{xy}
\usepackage{titling}
\predate{}
\postdate{}

\usepackage[explicit]{titlesec}
\usepackage{lipsum}

\usepackage{amsthm}
\newtheorem{thm}[subsection]{Theorem}
\newtheorem*{thm*}{Theorem}
\newtheorem{lem}[subsection]{Lemma}
\newtheorem{prop}[subsection]{Proposition}
\newtheorem{cor}[subsection]{Corollary}

\DeclareMathOperator{\MOD}{Mod}
\newcommand{\G}{\mathds{G}}
\newcommand{\Z}{\mathbb{Z}}
\newcommand{\Q}{\mathbb{Q}}
\newcommand{\F}{\mathbb{F}}
\newcommand{\Mod}{\MOD_G^{\sm}(k)}
\newcommand{\ModK}{\MOD_K^{\sm}(k)}
\DeclareMathOperator{\SL}{SL_2}
\newcommand{\B}{\mathcal{B}}
\renewcommand{\1}{\mathds{1}}

\renewcommand{\mod}{\MOD_G(k)}
\DeclareMathOperator{\GL}{GL}

\DeclareMathOperator{\Ind}{Ind}
\DeclareMathOperator{\cind}{c-Ind}
\DeclareMathOperator{\Hom}{Hom}
\newcommand{\Lim}{\varinjlim}
\newcommand{\plim}{\varprojlim}
\DeclareMathOperator{\sm}{sm}

\DeclareMathOperator{\Ext}{Ext}
 
\DeclareMathOperator{\ladm}{l adm}
\DeclareMathOperator{\lfin}{lfin}

\DeclareMathOperator{\Stab}{Stab}
\DeclareMathOperator{\cts}{cts}
\renewcommand{\O}{\mathcal{O}}
\DeclareMathOperator{\Ban}{Ban}
\DeclareMathOperator{\adm}{adm}
\newcommand{\Badm}{\Ban^{\adm}}

\newcommand{\OnK}{\O/\varpi^n\llbracket K \rrbracket }

\DeclareMathOperator{\fg}{fg}
\DeclareMathOperator{\m}{m}
\DeclareMathOperator{\E}{E}
\newcommand{\Ep}{\E[1/p]}

\DeclareMathOperator{\Ker}{Ker}
\DeclareMathOperator{\Sp}{Sp}
\DeclareMathOperator{\End}{End}

\newcommand{\Banb}{\Ban^{\adm}_{G,\zeta}(L)_\B}

\DeclareMathOperator{\mspec}{MaxSpec}

\DeclareMathOperator{\textP}{P}
\DeclareMathOperator{\oriented}{or}

\theoremstyle{definition}
\newtheorem{defi}[subsection]{Definition}

\theoremstyle{remark}
\newtheorem{rem}[subsection]{Remark}

\title{\vspace{-2cm}\textbf{Continuous Group Cohomology and Ext-Groups}}
\author{Paulina Fust}
\date{}
\begin{document}
	\maketitle
	\begin{abstract}
		We prove that the continuous group cohomology groups of a locally profinite group $ G $ with coefficients in a smooth $ k $-representation $ \pi $ of $ G $ are isomorphic to the $ \Ext $-groups $ \Ext^i_G(\1,\pi) $ computed in the category of smooth $ k $-representations of $ G $. We apply this to show that if $ \pi  $ is a supersingular $ \overline{\F}_p $-representation of $ \GL_2(\Q_p) $, then the continuous group cohomology of $ \SL(\Q_p)$ with values in $ \pi $ vanishes.
		
		Furthermore, we prove that the continuous group cohomology groups of a $ p $-adic reductive group $ G $, with coefficients in an admissible unitary $ \Q_p $-Banach space representation $ \Pi $, are finite dimensional. We show that the continuous group cohomology of $ \SL(\Q_p) $ with values in non-ordinary irreducible $ \Q_p $-Banach space representations of $ \GL_2(\Q_p) $ vanishes.
	\end{abstract}
	\section{Introduction} \label{Intro}
	
	Let $ k $ be a commutative ring with 1 and let $ G $ be a locally profinite group in the sense of \cite{bernshtein1976representations}, i.e.~a topological group which has a fundamental system of neighborhoods of the unit element consisting of compact open subgroups. 
	
	A smooth $ k $-representation of $ G $ is a $ k [G]$-module $ \pi $, such that the stabilizer $ \Stab_G(v) $ of any element $ v\in \pi $ is open in $ G $. Denote by $ \Mod $ the category of all smooth $ k $-representations of $ G $. We prove the following:
	 \begin{thm}[Corollary \ref{main result}]\label{Theorem Ext=H Intro}
	 	For any $ \pi \in \Mod $, and any $ i\ge 0 $, we have isomorphisms:
	\begin{equation}\label{Ext cong H intro}
	\Ext^i_G(\mathds{1},\pi)\cong H^i(G,\pi),
	\end{equation}
	where $ H^i(G,\pi) $ is the continuous group cohomology group of $ G $ with coefficients in $ \pi $ and the Ext-group is computed in the category $ \Mod $. 
	 \end{thm} 
 	The category $ \Mod $ is an abelian category which moreover has enough injectives by Proposition 2.1.1 in \cite{Emerton}. Therefore, the Ext-groups in $ \Mod $ can be defined by the right-derived functors of $ \Hom $. We prove Theorem \ref{Theorem Ext=H Intro} by showing that applying the functor of smooth vectors to the resolution of $ \pi $ used to compute the continuous group cohomology gives a resolution of $ \pi $ in $ \Mod $. This question was raised by Emerton in \cite[Section 2.2]{Emerton} and it is known in the case of compact groups (\cite[Proposition 2.2.6]{Emerton}).
	
	In Section \ref{section apllication}, we apply our result to the group $ G=\GL_2(\Q_p) $. Let $ \lvert - \rvert $ be a norm on $ \Q_p $, normalized so that $ \lvert p\rvert=1/p $. Then we can define a character $ \varepsilon:\Q_p^\times \rightarrow \Z_p^\times $, by $ x\mapsto x\lvert x \rvert $. Using Theorem \ref{Theorem Ext=H Intro}, we prove the following: \begin{thm}[Corollary \ref{H(SL2,pi)=0}]
		Let $ k $ be a finite field of characteristic $ p $ and let $ \pi\in\MOD^{\sm}_{\GL_2(\Q_p)} (k)$ be absolutely irreducible and not isomorphic to a twist by a character of $ \1 $, $ \Sp $ or $ \Ind_B^G\alpha $ , where $ \alpha:B\rightarrow k^\times  $ is defined by $ \begin{pmatrix}
		a&b\\0& d
		\end{pmatrix}\mapsto \varepsilon(ad^{-1}) \ \mathrm{ mod }\ p$. Then 
		 \[ H^i(\SL(\Q_p),\pi)=0 ,\]for $  i\ge0 $.
	\end{thm} 
	
	This includes for example all supersingular representations $ \pi\in\MOD^{\sm}_{\GL_2(\Q_p)} (k) $. The proof makes use of calculations of $ \Ext $-groups due to Pa\v{s}k\={u}nas (\cite{paskColmez}).
	This result is used by Colmez, Dospinescu and Nizioł in their forthcoming work.
	
	In Section \ref{section banach reps}, we study the continuous group cohomology groups of a $ p $-adic reductive group with coefficients in an admissible unitary $ L $-Banach space representation, where $ L /\Q_p$ is a finite extension. If $ \Pi $ is an admissible unitary $ L $-Banach space representation of $ G $ and $ \Pi^0 $ is an open bounded $ G $-invariant lattice in $ \Pi $, then we show that (Proposition \ref{H=limH}) for all $ i\ge 0 $ one has isomorphisms \[ H^i(G,\Pi)\cong (\plim_n H^i(G,\Pi^0/\varpi^n\Pi^0))[1/\varpi],\]
	for $ \varpi $ a uniformizer of $ L $. The proof uses the Bruhat--Tits building of $ G $ to obtain a resolution of the trivial representation of $ G $ by compactly induced representations from compact-mod-center subgroups of $ G $. Such resolutions appear in the work of Schneider--Stuhler \cite{schneider1997representation} and Casselman--Wigner \cite{CW}. Moreover, we deduce\begin{thm}[Corollary \ref{H^i(G,Pi) are finite diml}]
		 Let $ \Pi$ be an admissible unitary $ L $-Banach space representation of a $ p $-adic reductive group $ G $, then $ H^i(G,\Pi) $ is finite dimensional over $ L $, for all $ i\ge 0 $.
	\end{thm} In the case where $ G $ is a compact $ p $-adic analytic group, these isomorphisms follow directly from \cite{emerton2006interpolation}.
	
	As in Section \ref{section apllication}, we apply these results to the group $ \GL_2(\Q_p) $ and obtain a similar statement for Banach space representations:
	\begin{thm}[Proposition \ref{H(SL2,Pi)=0}]
		Any absolutely irreducible admissible unitary $ L $-Banach space representation $ \Pi $ of $ \GL_2(\Q_p) $, which is not isomorphic to a twist by a unitary character of $ \1 $, $ \widehat{\Sp} $ or $  \Ind^G_B\tilde{\alpha} $, has trivial continuous group cohomology groups over $ \SL(\Q_p) $, i.e.\[ H^i(\SL(\Q_p),\Pi)=0 ,\] for $ i\ge0 $.
	\end{thm}
	Here, $ \tilde{\alpha} :B\rightarrow L^\times $ is the representation of $ B $, defined by $ \begin{pmatrix}
	a&b\\0& d
	\end{pmatrix}\mapsto \varepsilon(ad^{-1})  $.
	
		\textit{Acknowledgements}: I want to thank my PhD supervisor Vytautas Pa\v{s}k\={u}nas for his support and Gabriel Dospinescu, Jessica Fintzen, Lennart Gehrmann and the referee for their helpful comments and suggestions. This work was funded by the DFG Graduiertenkolleg 2553.
	
	\section{Continuous group cohomology}
	
	Following \cite{CW}, we define the continuous group cohomology as follows:
	Let $ V $ be a topological $ G $-module, i.e. a topological abelian group $ V $ with a $ G $-action such that $ G $ acts on $ V $ via group automorphisms and the map $ G\times V\rightarrow V $ is continuous. 
	Then we can define the cochain complex $$ C^n(G,V):= C(G^{n+1},V) :=\{f:G^{n+1}\rightarrow V \text{ continuous}\},$$ with differentials $ d^n:C^{n}(G,V)\rightarrow C^{n+1}(G,V) $, defined by $$
		d^nf(g_0,\dots,g_{n+1})=\sum_{i=0}^{n+1}(-1)^i f(g_0,\dots , \widehat{g_i},\dots, g_{n+1}).$$

		We endow the spaces $ C^n(G,V) $ with the compact-open topology. By defining a $ G $-action via $ (gf)(g_0,\dots,g_n)=gf(g^{-1}g_0,\dots ,g^{-1}g_n) $, for $ g,g_0,\dots,g_n\in G $ and $ f\in C^n(G,V) $, these will define topological $ G $-modules. Furthermore, there is a continuous $ G $-equivariant injection $ V\hookrightarrow C(G,V) $, defined by $ v\mapsto [g\mapsto v] $.\\
	
	For a closed subgroup $ H\le G $ and a topological $ G $-module $ V $, define the induced representation to be $$ \Ind_H^GV:=\{f\in C(G,V) \mid f(hg)=hf(g)\ \forall h\in H,\ \forall g\in G\} ,$$ 
	with the subspace topology induced from $ \Ind^G_HV\subseteq C(G,V) $ and $ G $-action $ gf(g'):=f(g'g) $, for $ g,g'\in G $. Similarly, the compact induction is defined as $$ \cind_H^GV:=\{f\in \Ind_H^G V\mid  \text{ the support of } f \text{ is compact modulo }H\}  $$ with the same $ G $-action.
	
	\begin{lem} \label{Cn=C0}
		For a topological $ G $-module $ V $, there are homeomorphisms of $ G$-modules \begin{equation*}
		C^{n+1}(G,V) \cong C^0(G,C^n(G,V)) \text{ and } C^0(G,V)\cong \Ind^G_1V,
		\end{equation*}
		for all $ n\ge 0 $.
	\end{lem}
	
	\begin{proof}
		For the first homeomorphism, see \cite[X.3.4 Corollaire 2]{Bourbaki}.
		
		The assignment $ f\mapsto  [g\mapsto gf(g^{-1})]$ defines both maps $ C^0(G,V)\rightarrow \Ind^G_1V  $ and its inverse. One can check that both maps are continuous and $ G $-equivariant, hence the modules are homeomorphic.
	\end{proof}

	\begin{lem}\label{Frobenius}
		For two topological $ G $-modules $ V $ and $ W $, one has the following isomorphism: \[\Hom_G^{\cts}(V,\Ind^G_1W)\cong \Hom^{\cts}(V,W), \]
		where $ \Hom^{\cts}(V,W) $ are continuous group homomorphisms and on the left hand side, we take continuous $ G $-equivariant group homomorphisms.
	\end{lem}

	\begin{proof}
		\cite[Lemma 2]{CW}.
	\end{proof}

	\begin{lem}
		The complex $ 0\rightarrow V\rightarrow C^\bullet(G,V)  $ is an exact complex of $ G $-modules.
	\end{lem}

	\begin{proof}
		Let $ C^\bullet $ be the complex with $ C^{-1}=V $, $ C^i=0 $, for $ i\le -2 $ and $ C^i=C^i(G,V) $ for $ i\ge 0 $. 
		To prove the exactness of this complex, we construct a cochain homotopy 
		 $ s^n:C^n\rightarrow C^{n-1} $, $ n\in\Z $ between $ id_{C^\bullet} $ and the zero map on $ C^\bullet $. Define $ (s^nf)(g_1,\dots , g_n):= f(1,g_1,\dots,g_n) $ for $ f\in C^n(G,V) ,\ n\ge 0$ and $ s^n=0 $ for $ n\le -1 $. Then $ s^nf :G^{n}\rightarrow V$ is a continuous map, since it is the composition of the continuous maps $ f $ and $ \{1\}\times G^n\hookrightarrow G\times G^n $. Moreover, one can easily check that it satisfies $ s^{n+1}d^n+d^{n-1}s^n=id_{C^n} $ for all $ n $. Hence we have found the desired homotopy and the complex is exact.
	\end{proof}

	\begin{defi}
		The $ i^\text{th} $ continuous group cohomology group of $ G $ with coefficients in $ V $ is defined to be $ H^i(G,V):= H^i(C^\bullet(G,V)^G) $.
	\end{defi} 
	
	A homomorphism of topological $ G $-modules $ \phi:V\rightarrow W $ is said to be a strong morphism, if the induced morphisms $ \Ker \phi\rightarrow V $ and $ V/\Ker\phi\rightarrow W $ each have a continuous left inverse in the category of topological abelian groups. And if one has a short exact sequence of topological $ G $-modules \[ 0\rightarrow V\rightarrow W \rightarrow U\rightarrow 0, \] in which all the morphisms are strong, then the induced sequence \[ 0\rightarrow C^n(G,V)\rightarrow C^n(G,W)\rightarrow C^n(G,U)\rightarrow 0  \] is again strong and then induces a long exact sequence in cohomology. 
	
	For example, any short exact sequence\[ 0\rightarrow V\rightarrow W \rightarrow U\rightarrow 0, \]of topological $ G $-modules is automatically strong if $ U $ is discrete. Therefore it  induces a long exact sequence in cohomology. 
	
	\begin{defi}
		We say that a topological $ G $-module $ V $ is acyclic (for the continuous group cohomology), if  \[ H^i(G,V)=\begin{cases}
		V^G, \text{ if } i=0\\
		0,\text{ otherwise}.
		\end{cases}  \]
	\end{defi}
	
		\begin{rem}
			\begin{enumerate}
				\item The $ G $-modules $ C^n(G,V) $ are acyclic for the continuous group cohomology. Hence, the complex $ 0\rightarrow V\rightarrow C^\bullet(G,V) $ is an acyclic resolution of $ V $. (cf. \cite[p. 201]{CW}.)
				\item A topological $ G $-module $ V $ is said to be continuously injective, if for every strong $ G $-injection $ U\hookrightarrow W $ and for every morphism of topological $ G $-modules $ \phi:U\rightarrow V $, the morphism $ \phi $ extends to a morphism from $ W $ to $ V $.
				An example of continuously injective $ G $-modules are the modules $ C^n(G,V) $ (cf. \cite[p. 201]{CW}).
				\item If we equip a smooth representation $ \pi\in\Mod $ with the discrete topology, this will give a topological $ G $-module and we can define the continuous group cohomology groups $ H^i(G,\pi) $ with coefficients in $ \pi $.
			\end{enumerate}
		\end{rem}
	
	\begin{lem}\label{wlog center acts trivially}
		Let $ V $ be a topological $ G $-module which also has the structure of a topological $ R $-module for some topological ring $ R $, such that the $ G $-action is $ R $-linear. Assume that an element $ z\in Z(G) $ in the center of $ G $ acts on $ V $ by multiplication with a scalar $ \lambda\in R $. If $ 1-\lambda$ is a unit in $ R $, then $ H^i(G,V)=0 $ for all $ i\ge 0 $.
	\end{lem}

\begin{proof}
	We consider the continuous $ R $-linear $ G $-module homomorphism \[ \m_z:V\rightarrow V,\ v\mapsto zv .\] We can extend it to a map of complexes
	\begin{equation*}
	\begin{xy}
	\xymatrix{
		0 \ar[r] &V \ar[d]^{\m_z}\ar[r] &C^0(G,V) \ar[d]^{\m_z}\ar[r]& C^1(G,V) \ar[d]^{\m_z}\ar[r] &\dots \\
		0 \ar[r] &V \ar[r] &C^0(G,V) \ar[r] &C^1(G,V) \ar[r] &\dots
	}
	\end{xy} 
	\end{equation*}
	where $ \m_z $ is the multiplication by $ z $. Since $ z $ is central, $ \m_z $ is $ G $-equivariant.
	
	By assumption, $ z $ acts on $ V $ by mutiplication with the scalar $ \lambda $, hence we can also construct the following map of complexes:
	 \begin{equation*}
	\begin{xy}
		\xymatrix{
			0 \ar[r] &V \ar[d]^{\lambda}\ar[r] &C^0(G,V) \ar[d]^{\lambda}\ar[r]& C^1(G,V) \ar[d]^{\lambda}\ar[r] &\dots \\
			0 \ar[r] &V \ar[r] &C^0(G,V) \ar[r] &C^1(G,V) \ar[r] &\dots
		}
	\end{xy} 
	\end{equation*}
	where the vertical arrows are given by multipication with $ \lambda $.
	
	Therefore, we have two continuous, $ G $-equivariant endomorphisms of the complex $ 0\rightarrow V\rightarrow C^\bullet(G,V) $ extending the map $ \m_z $. But this complex is a strong resolution of $ V $ by continuously injective $ G $-modules and by \cite[Section 2]{hochschild1962cohomology}, the map of $ \m_z $ extends uniquely up to homotopy. Hence, the map of complexes $ \m_z-\lambda $ is null-homotopic and induces the zero map on the cohomology groups $ H^i(G,V) $ for all $ i\ge 0 $. But passing to $ G $-invariants, the map $ \m_z-\lambda $ is just $ id_{C^\bullet(G,V)} - \lambda $. In conclusion, multiplication with $ 1-\lambda $ is the zero map on $ H^i(G,V) $, so that the continuous group cohomology groups $ H^i(G,V) $ must be zero if $ 1-\lambda $ is a unit in $ R $.
\end{proof}
	
	\section{Proof for profinite groups}
	Proposition \ref{compact open} below can already be found in Section 2.2 of \cite{Emerton}. To make the article as self-contained as possible we decided to include a proof.\\
	
	First, we consider the case of a compact locally profinite group $ K $, i.e. a profinite group, and a smooth representation $ \pi\in\ModK $ of $ K $ equipped with the discrete topology, so that we can define $ H^i(K,\pi) $ as above.
	Moreover, we denote by $ \1\in\ModK $ the free $ k $-module of rank 1 with trivial $ K $-action and $ \Ext^i_K(\1,\pi) $ the $ \Ext $-group computed in $ \ModK $, i.e. the $ i^\text{th} $ right derived functor of $ \Hom_K(\1,-) $ applied to $ \pi $.\\ 
	The following lemmas will be used in several arguments throughout the paper:
	\begin{lem}[{\cite[Proposition 2.3.10]{weibel1995introduction}}]\label{lemma:exact right adjoint preserves injectives}
		Let $ R:\mathcal{D}\rightarrow\mathcal{C}  $ be an additive functor and suppose that it is right adjoint to an exact functor $ L:\mathcal{C}\rightarrow\mathcal{D} $. Then the functor $ R $ preserves injective objects.
	\end{lem}

	\begin{lem}\label{Vigneras Lemma on Ext groups}
		Let $ L:\mathcal{C}\rightarrow \mathcal{D} $ be an exact functor between abelian categories and suppose that $ L $ admits an exact right adjoint $ R:\mathcal{D}\rightarrow \mathcal{C} $. Assume that $ \mathcal{D} $ has enough injectives. Then there are natural (in $ C $ and $ D $) isomorphisms \[ \Ext^i_{\mathcal{D}}(L(C),D)\cong \Ext^i_{\mathcal{C}}(C,R(D)), \text{ for all }i\ge 0. \]
	\end{lem}

	\begin{proof}
		By assumption, $ \mathcal{D} $ has enough injectives, so that we can find an injective resolution $ D\hookrightarrow I^\bullet $ of $ D $ in $ \mathcal{D} $. Moreover, since both functors are exact and $ R $ is right adjoint to $ L $, applying $ R $ gives an injective resolution $ R(D)\hookrightarrow R(I^\bullet) $ of $ R(D) $ in $ \mathcal{C} $. We get\begin{align*}
			\Ext^i_{\mathcal{D}}(L(C),D)&=H^i(\Hom_{\mathcal{D}}(L(C),I^\bullet))\\
			&\cong H^i(\Hom_{\mathcal{C}}(C,R(I^\bullet))) \\
			&=\Ext^i_{\mathcal{C}}(C,R(D)).
		\end{align*}
	\end{proof}

	\begin{prop}\label{compact open}
		One has isomorphisms
		\begin{equation*}
		\Ext^i_K(\mathds{1},\pi)\cong H^i(K,\pi),
		\end{equation*}
		for all $ i\ge 0 $.
	\end{prop}

	We split the proof into the following lemmas.\\
	
	For a topological $ K $-module $ V $, we define the space of smooth functions $ C^{\sm}(K,V)\subseteq C(K,V) $ to be: $$C^{\sm}(K,V):=\{f:K\rightarrow V\mid \exists H\le K\text{ open, s.t. } hf=f\ \forall h\in H\}.$$
	
	\begin{lem} \label{cts=sm}
		If V carries the discrete topology, then we have the following equality: \begin{align*}
		C(K,V)&=C^{\sm}(K,V).
		\end{align*}
	\end{lem}
	
	\begin{proof}
		The inclusion $ C^{\sm}(K,V)\subseteq C(K,V) $ is trivial. Conversely, if $ f:K\rightarrow V $ is a continuous function, it is locally constant and, since $ K $ is compact, the image of $ f $ consists of finitely many elements $ v_1,\dots,v_n\in V $. Moreover, for each $ h\in K $, there is an open subgroup $ U_h\le K $, such that $ U_hh\subseteq f^{-1}(v_i) $ for some $ i $. We get \[  K=\underset{h\in K}{\bigcup } U_hh=\underset{i=1}{\overset{m}{\bigcup}}U_{h_i}h_i . \] Thus, $ U:=\underset{i}{\bigcap} U_{h_i}\cap \text{Stab}_K(f(h_i)) $ is an open subgroup of $ K $ with the property that for any $ u\in U $, $ uf=f$. 
	\end{proof}

	\begin{lem}
		If the topology of $ V $ is discrete, the compact-open topology on $ C(K,V) $ is discrete. 
	\end{lem}
	\begin{proof}
		Let $ f:K\rightarrow V $ be an element in $ C(K,V)=C^{\sm}(K,V) $, i.e. there exists a compact open subgroup $ H\le K $, such that $ f(gh)=f(g) $ for all $ h\in H,\ g\in K $. Since $ K $ is compact, $ K $ is the disjoint union of finitely many compact open cosets $ g_iH $, $ i=1,\dots, n $,  on which $ f $ is constant. Then $ \{f\}=\underset{i=1}{\overset{n}{\bigcap}}\Omega(g_iH,\{f(g_i)\}) $ is open as it is the finite intersection of the open sets \[  \Omega(g_iH,\{f(g_i)\}) =\{f'\in C(K,V)\mid \ f'(g_iH)\subseteq \{f(g_i)\}\} \] in $ C(K,V) $.
	\end{proof}

	\begin{lem}\label{acyclic resolution} 
		If $ \pi \in\ModK$ is a smooth representation of $ K$ equipped with the discrete topology, then
		$ 0\rightarrow \pi \rightarrow C^\bullet(K,\pi) $ is an acyclic resolution of $ \pi  $ in the category of smooth representations $ \ModK $.
	\end{lem}

	\begin{rem}
		If we say that a topological $ K $-module is acyclic, we mean that it is acyclic for the continuous group cohomology. Whereas acyclic in $ \ModK $ means acyclic with respect to $ \Ext^i_K(\1,-) $.
	\end{rem}
	
	\begin{proof}
		We know that the complex is exact and that $ C^n(K,\pi)$ is a smooth representation by the previous two lemmas. The representations $C^n(K,\pi)\cong C(K,C^{n-1}(K,\pi)) $ are acyclic in the category $ \ModK $: Indeed, since $V:=C^{n-1}(K,\pi)  $ is again discrete, by Frobenius reciprocity, one has $ \Hom_K(\1,C(K,V))\cong \Hom_k(k,V)\cong V $. Moreover, the functor $ C(K,-) $ is exact and we can apply Lemma \ref{Vigneras Lemma on Ext groups} to get $ \Ext^i_K(\1,C(K,V))\cong \Ext^i_k(k,V) $, which is zero for $ i>0 $ and $ V $ for $ i=0 $, as expected.
	\end{proof}

	\begin{proof}[Proof of Proposition \ref{compact open}]
		By Lemma \ref{acyclic resolution}, we can compute the groups $ \Ext^i_K(\1,\pi) $ by using the acyclic resolution \[  0\rightarrow \pi \rightarrow C^\bullet(K,\pi).\] More precisely, it is the $ i^\text{th} $ cohomology group of the complex \begin{equation*}
		\Hom_K(\1,C^\bullet(K,\pi)) \cong C^\bullet(K,\pi)^K
		\end{equation*}
		and this is the same as the continuous group cohomology group $ H^i(K,\pi) $. This proves the proposition.
	\end{proof}

	\section{The general case}
	
	If $ G $ is locally compact, but not compact, then the modules $ C^\bullet(G,\pi) $ will not lie in $ \Mod $. 
For example, if $ G=(\Q_p,+) $, then the indicator function on the set $ \bigcup_{n\ge 1}(p^{-n}+p^n\Z_p) $ is locally constant and hence continuous. But it is not smooth, since its stabilizer consists of $ x\in\Q_p $, such that $ x\in p^n\Z_p $ for all $ n\ge 1 $, and hence is zero, which is not open.\\
	
	To conclude in this case, we will use the following functor from the category $ \mod $ of $ k $-representations of $ G $ to the category $ \Mod $ \begin{align*}
	(-)^{\sm}:\mod \rightarrow \Mod, \\
	V\mapsto V^{\sm}:=\Lim_{K}V^K,
	\end{align*}
	where the direct limit is taken over the directed family of compact open subgroups $ K\le G $. Hence, $ V^{\sm} $ can be seen as the subrepresentation of $ V $ consisting of the smooth vectors. Moreover, this functor is right-adjoint to the inclusion $ \Mod\hookrightarrow \mod $, since $\Hom_G(W,V)= \Hom_{G}(W,V^{\sm}) $ for a smooth representation $ W\in\Mod $. Using Lemma \ref*{lemma:exact right adjoint preserves injectives}, this implies that the functor $ (-)^{\sm} $ preserves injective objects.
	
	\begin{lem}\label{computes H(K)}
		For a compact open subgroup $ K\le G $ and a smooth representation $ \pi\in \Mod $, the $ G $-modules $
		 C^n(G,\pi) $ are acyclic for the continuous group cohomology $ H^\bullet(K,-) $. 
		 
		Moreover, the cohomology of the complex $ C^\bullet(G,\pi)^K $ is $ H^\bullet(K,\pi) $.
	\end{lem}
	
	\begin{proof}
		A proof is given in \cite[Proposition 4 (a)]{CW}. The idea is as follows: since the quotient $ G/K $ is discrete, one has a continuous section $ s:G/K\rightarrow G $ of the natural projection $ G\rightarrow G/K $ with which one can define the $ K $-invariant homeomorphism $ K\times G/K\rightarrow G,\ (h,gK)\mapsto s(gK)h $. Therefore, one gets an isomorphism of $ K $-modules $$ C(G,\pi)\cong C(K\times G/K,\pi)\cong C(K,C(G/K,\pi)) $$ showing that $ C(G,\pi) $ is acyclic as a $ K $-module.
		
		To conclude that the complex $ C^\bullet(G,\pi)^K $ computes the cohomology of $ K $, one also has to use that the resolution $ \pi\hookrightarrow C^\bullet(G,\pi) $ is a strong resolution of acyclic $ K $-modules and hence can be used to compute the cohomology of $ K $ (\cite[Proposition 1]{CW}). 
	\end{proof}

	\begin{prop}
		Let $ \pi\in \Mod $ be a smooth $ k $-representation of $ G $. Applying the functor $ (-)^{\sm} $ to the complex $ C^\bullet(G,\pi) $ gives a resolution $ 0\rightarrow\pi\rightarrow C^\bullet(G,\pi)^{\sm} $ of $ \pi  $ in $ \Mod $.
	\end{prop}
	
	\begin{proof}
		We want to show that the cohomology $ H^i(C^\bullet(G,\pi)^{\sm}) $ vanishes for $ i>0 $. By definition, $ C^n(G,\pi)^{\sm} $ is the direct limit over all compact open subgroups $ K \le G$ of $ K $-fixed vectors $ C^n(G,\pi)^K $. Note that taking direct limits in $ \Mod $ is exact and hence commutes with taking cohomology of a complex. In particular, we get $$ H^i(C^\bullet(G,\pi)^{\sm}) \cong \Lim_{K} H^i(C^\bullet(G,\pi)^K) . $$
		By Lemma \ref{computes H(K)}, this is the same as $ \Lim_{K} H^i(K,\pi)$. Now, this direct limit vanishes, since any element in $ \Lim_{K} H^i(K,\pi)$ comes from an element $ f\in H^i(K,\pi) $ for some compact open subgroup $ K\le G $. By Lemma \ref{cts=sm}, this $ f $ is represented by a smooth, hence locally constant cocycle, hence it vanishes when restricted to some compact open subgroup $ K'\le K $. But since the transition maps in the direct limit are given by restrictions, this element is zero in the direct limit. Hence, $ H^i(C^\bullet(G,\pi)^{\sm}) \cong \Lim_{K} H^i(C^\bullet(G,\pi)^K)=0 $ and $ C^\bullet(G,\pi)^{\sm} $ gives indeed a resolution of $ \pi $. 
	\end{proof}

	\begin{lem} \label{ExtG=Extk}
		For all $ \pi ,\ V\in \Mod$, we have isomorphisms \[ \Ext^i_G(\pi, C^n(G,V)^{\sm})\cong \Ext^i_k(\pi,C^{n-1}(G,V)), \text{ for all }n\ge 0,\] where the $ \Ext $-group on the right hand side is computed in the category of $ k $-modules and $ C^{-1}(G,V) $ is defined to be $ V $.
		
		In particular, $ \Ext^i_G(\1,C^n(G,V)^{\sm})=0 $ for all $ i\ge 1 $.
	\end{lem}

	\begin{proof}
			By Lemma \ref{Cn=C0}, we have an isomorphism $ C(G,-)^{\sm}\cong (\Ind^G_1(-))^{\sm}$. By Frobenius reciprocity, we have \[ \Hom_{G}(\pi,C(G,V)^{\sm})\cong  \Hom_{G}(\pi,(\Ind_1^G(V))^{\sm})\cong \Hom_k(\pi,V), \] for any $ \pi, V\in \Mod. $
			In particular, the functor $ C(G,-)^{\sm} $ is right adjoint to the restriction functor which is exact. Moreover, $ C(G,-)^{\sm} $ is left-exact, and it is even exact. Indeed, if $ U\overset{\alpha}{\twoheadrightarrow} W $ is a surjection in $ \Mod $, then the induced map $ C(G,U)^{\sm}\rightarrow C(G,W)^{\sm} $ is also surjective. To see this, note that for any smooth function $ f:G\rightarrow W $, there is a compact open subgroup $ K $ of $ G $, such that $ f $ is constant on cosets $ Kg $. Then we can choose a set of representatives $ g_i ,\ i\in I$ of $ K\backslash G $ and choose for each $ g_i $ an element $ u_i\in U $, such that $ \alpha(u_i)=f(g_i) $ and define a map in $ C(G,U)^{\sm} $ by sending an element in the coset $ Kg_i $ to the element $ u_i $. This will map to $ f $.
			
			Hence, we can apply Lemma \ref{Vigneras Lemma on Ext groups} to get isomorphisms \[ \Ext^i_G(\pi, C(G,V)^{\sm})\cong\Ext^i_k(\pi,V). \]
			The claim follows from the description $ C^n(G,V)\cong C(G,C^{n-1}(G,V)) $ from Lemma \ref{Cn=C0}.
			 		
			 		Moreover, if $ \pi=\1 $, then $ \Hom_k(\1,I^\bullet)\cong I^\bullet $, thus $$\Ext^i_G(\1,C^n(G,V)^{\sm})\cong \Ext^i_k(\1,C^{n-1}(G,V))=0 $$ for all $ i>0. $
	\end{proof}
	
	\begin{cor} \label{main result}
		Let $ \pi\in\Mod $. We have isomorphisms of cohomology groups \begin{equation} \label{isomorphism}
		\Ext^i_G(\mathds{1},\pi)\cong H^i(G,\pi),
		\end{equation}
		for all $ i\ge 0 $.
	\end{cor}
	
	\begin{proof}
		Just note that $ (C^\bullet(G,\pi)^{\sm})^G \cong C^\bullet(G,\pi)^G$ and by Lemma \ref{ExtG=Extk}, $ C^\bullet(G,\pi)^{\sm} $ is an acyclic resolution of $ \pi $ in $ \Mod $, hence $ (C^\bullet(G,\pi)^{\sm})^G $ computes $ \Ext^i_G(\1,\pi) $, whereas the cohomology in degree $ i $ of $ C^\bullet(G,\pi)^G $ is of course $ H^i(G,\pi) $.
	\end{proof}
	
	\section{Mod $ p $ representations of $ \GL_2(\Q_p) $} \label{section apllication}
	From now on, let $ G=\GL_2(\Q_p) $, let $ k $ be a finite field of characteristic $ p $. Let $ Z $ be the center of $ G $, $ \1 \in\Mod$ be the trivial representation, $ \Sp $ be the Steinberg representation and let $ \omega:\Q_p^\times\rightarrow k^\times $ be the character $ \omega(x)=x\vert x\vert \ (\mathrm{mod}\ p)$, where the absolute value $ \lvert - \rvert $ is normalized so that $ \lvert p \rvert=1/p $. This character induces a representation of the subgroup $ B\le G $ of upper triangular matrices, $ \alpha:B\rightarrow k^\times,  $ defined by $ \begin{pmatrix}
	a&b\\0& d
	\end{pmatrix}\mapsto \omega(a)\omega(d)^{-1} $. And we denote by $ \Ind_B^G\alpha $ the induced representation, given by continuous functions $ f:G\rightarrow k $ with $ f(bg)=\alpha(b)f(g) $ for all $ b\in B $ and $ g\in G $. 
	
	Moreover, we fix a continuous character $ \zeta :Z\rightarrow \O^\times $ and we define the category $ \MOD^{\sm}_{G,\zeta}(k) $ to be the full subcategory of $ \MOD^{\sm}_{G}(k) $ consisting of representations on which $ Z $ acts by $ \zeta $ and we define the category $\MOD^{\lfin}_{G,\zeta} (k)  $ to be the full subcategory of $ \MOD^{\sm}_{G,\zeta}(k) $ consisting of those representations that are locally admissible, or equivalently, that are locally of finite length (cf \cite[Section 5.4]{paskColmez}). If $ \mathcal{O} $ is the ring of integers of some finite field extension $ L/\Q_p $, such that its residue field is equal to $ k $, then one can similarly define the category $\MOD^{\lfin}_{G,\zeta} (\mathcal{O})  $ and $\MOD^{\lfin}_{G,\zeta} (k)  $ forms a subcategory of it. Moreover, by \cite[Proposition 5.34]{paskColmez}, this category decomposes into a direct product of subcategories $$ \MOD^{\lfin}_{G,\zeta} (\mathcal{O})\cong\prod_{\mathcal{B}}\MOD^{\lfin}_{G,\zeta}(\mathcal{O})_\mathcal{B} .$$
	(Note that in \cite{paskColmez}, these subcategories are denoted by $ \MOD^{\lfin}_{G,\zeta}(\mathcal{O})^{\mathcal{B}}  $.) Therefore, as a subcategory, $ \MOD^{\lfin}_{G,\zeta}(k) $ will also decompose accordingly. These blocks $ \mathcal{B} $ are certain equivalence classes of irreducible representations in $ \MOD^{\sm}_{G,\zeta}(k) $. For a precise definition, see Section 5.5 in \cite{paskColmez}. Then, for such block $ \B $, the corresponding full subcategory $  \MOD^{\lfin}_{G,\zeta}(\mathcal{O})_\mathcal{B} $ of $ \MOD^{\lfin}_{G,\zeta}(\mathcal{O}) $ is given by those representations $ \pi\in \MOD^{\lfin}_{G,\zeta}(\mathcal{O}) $, such that all its irreducible subquotients lie in the block $ \B $.
	
	We can use the isomorphism \eqref{isomorphism} to prove that the continuous group cohomology group $ H^i(\SL(\Q_p),\pi) $ for some smooth irreducible representations $ \pi $ is trivial. 
	
	\begin{prop}\label{ext_Z=0}
		Let $ \pi \in\Mod$ be a smooth absolutely irreducible representation of $ G $, with central character $ \zeta_\pi $. Assume that for any character $ \chi:\Q_p^\times\rightarrow k^\times $, $ \pi\otimes \chi\circ\det $ is not isomorphic to $ \1 $, $ \Sp $ or $ \Ind_B^G\alpha $, then one has \begin{equation*}
		\Ext^i_{G,\zeta_\pi}(\Ind^G_{Z\SL(\Q_p)}\zeta_\pi,\pi)=0,
		\end{equation*}
		for all $ i\ge 0 $, where the $ \Ext $-group is computed in the category $ \MOD^{\sm}_{G,\zeta_\pi}(k) $ of smooth representations of $G $ with central character $ \zeta_\pi $.
	\end{prop}
	
	\begin{proof}
		Note that if $ p\neq 2 $, $ Z\SL(\Q_p) $ is of index 4 in $ G $ and if $ p=2 $ it is of index 8, so that the representation $ \Ind^G_{Z\SL(\Q_p)}\zeta_\pi $ is 4- or 8-dimensional. The action of $ \GL_2(\Q_p) $ on $ \Ind^G_{Z\SL(\Q_p)}\zeta_\pi $ factors through an abelian quotient, so that, after enlarging the field $ k $, we may assume that there exists a character 
		$ \eta_1 $, a 3-, respectively 7-dimensional representation $ \tau \in \Mod  $ and a short exact sequence of the form $$ 0\rightarrow \eta_1\circ\det \rightarrow \Ind^G_{Z\SL(\Q_p)}\zeta_\pi \rightarrow \tau \rightarrow 0.$$
		This induces the long exact sequence \begin{equation*}
		\dots \rightarrow \Ext^i_{G,\zeta_\pi}(\tau,\pi) \rightarrow\Ext^i_{G,\zeta_\pi}(\Ind^G_{Z\SL(\Q_p)}\zeta_\pi,\pi) \rightarrow \Ext^i_{G,\zeta_\pi}(\eta_1\circ \det,\pi) \rightarrow \dots
		\end{equation*}
		In particular, it is enough to show that $ \Ext^i_{G,\zeta_\pi}(\eta_1\circ \det, \pi)=0 $ and $ \Ext^i_{G,\zeta_\pi}(\tau, \pi)=0 $. Moreover, we can repeat this procedure and find $ \eta_2\circ\det \hookrightarrow\tau$ and so on. Inductively, it will be enough to show that $ \Ext^i_{G,\zeta_\pi}(\eta\circ \det, \pi)=0  $ for all such characters $ \eta $.
		
		The representation $ \eta\circ \det$ is clearly irreducible. Moreover, $ \pi $ and $ \eta\circ \det $ lie in the full subcategory $ \MOD^{\ladm}_{G,\zeta_\pi} (k) =\MOD^{\lfin}_{G,\zeta_\pi} (k)$ of $ \MOD^{\sm}_{G,\zeta_\pi}(k) $. Moreover, by Corollary 5.17 of \cite{paskColmez}, we have isomorphisms $ \Ext^i_{G,\zeta_\pi}(\eta\circ \det, \pi)\cong  \Ext^{\lfin,i}_{G,\zeta_\pi}(\eta\circ\det, \pi)$, meaning that we can compute the $ \Ext  $-groups in the category $  \MOD^{\lfin}_{G,\zeta_\pi} (k)$, which decomposes into a direct product of subcategories $$ \MOD^{\lfin}_{G,\zeta_\pi} (k)\cong\prod_{\mathcal{B}}\MOD^{\lfin}_{G,\zeta_\pi}(k)_\mathcal{B} $$ (\cite[Proposition 5.34]{paskColmez}). In particular, there are no extensions between representations lying in different blocks $ \mathcal{B} $. The blocks are described in Corollary 1.2 of \cite{pavskunas2014blocks} and by assumption, $ \pi $ does not lie in the same block of any character. Hence, $ \Ext^{\lfin,i}_{G,\zeta_\pi}(\eta_j\circ\det, \pi)=0 $ for all $ i\ge 0 $ and the claim follows.
	\end{proof}

	\begin{cor}	\label{H(SL2,pi)=0}
		Under the assumptions of Proposition \ref{ext_Z=0} we have
		\begin{equation*}
			H^i(\SL(\Q_p),\pi) =0,\ \forall i\ge 0.
		\end{equation*}
	\end{cor}
	\begin{proof}
		We may assume that $ Z\cap \SL(\Q_p) $ acts trivially on $ \pi $, since otherwise the continuous group cohomology groups $ H^i(\SL(\Q_p),\pi) $ are zero for all $ i\ge 0 $ by Lemma \ref{wlog center acts trivially}.
		By Corollary \ref{main result}, we know that $ H^i(\SL(\Q_p),\pi)\cong \Ext^i_{\SL(\Q_p)}(\1,\pi)$. 
		We can then consider the exact functor $ \MOD_{\SL(\Q_p),\1}^{\sm}(k)\rightarrow\MOD_{Z\SL(\Q_p),\zeta_\pi}^{\sm}(k) $, given by extending the action of $ \SL(\Q_p) $ to $ Z\SL(\Q_p) $, by letting $ Z $ act via the character $ \zeta_\pi $. This is left adjoint to the restriction functor which is also exact. In particular, by Lemma \ref{Vigneras Lemma on Ext groups}, we obtain an isomorphism $ \Ext^i_{\SL(\Q_p)}(\1,\pi)\cong \Ext^i_{Z\SL(\Q_p),\zeta_\pi}(\zeta_\pi,\pi)$ for all $ i\ge 0  $. The functor \[ \Ind_{Z\SL(\Q_p)}^G(-):\MOD_{Z\SL(\Q_p),\zeta_\pi}^{\sm}(k)\rightarrow \MOD^{\sm}_{G,\zeta_\pi}(k)\] is exact and has an exact right adjoint functor given by the restriction. Therefore, we can apply Lemma \ref{Vigneras Lemma on Ext groups} to conclude that  \[ \Ext^i_{Z\SL(\Q_p),\zeta_\pi}(\zeta_\pi,\pi)\cong\Ext^i_{G,\zeta_\pi}(\Ind^G_{Z\SL(\Q_p)}\zeta_\pi,\pi), \] which is zero by Proposition \ref{ext_Z=0}.
	\end{proof}
	
	\begin{rem}
		For any field extension $ l $ of $ k $, one can use Proposition 5.33 in \cite{paskColmez}, to see that the blocks are the same over $ l $. In particular, we have the same vanishing of $ \Ext $-groups stated in Corollary \ref{H(SL2,pi)=0}.
	\end{rem}
	
	\section{Banach space representations} \label{section banach reps}
	
	We want to apply Corollary \ref{main result} to Banach space representations of a $ p $-adic reductive group $ G $. More precisely, we let $ F/\Q_p $ be a finite extension and $ \G $ be a connected reductive group over $ F $. Then take $ G =\G(F)$ to be the group of its $ F $-rational points. Let $ L $ be a finite extension of $ \Q_p $ with ring of integers $ \O $, uniformizer $ \varpi $ and residue field $ k $. Denote $ \Badm_G(L) $ the category of admissible unitary $ L $-Banach space representations of $ G $. These are representations $ \Pi $ of $ G $ on $ L $-vector spaces which are complete with respect to a $ G $-invariant norm and which are admissible in the sense of Schneider--Teitelbaum \cite{schneider2002banach}. By Proposition 9.6 in \cite{SCHNEIDERp-adicBanach}, this means that there is an open bounded $ G $-invariant lattice $ \Pi^0 $ in $ \Pi $, such that the quotient $ \Pi^0/\varpi\Pi^0 $ is an admissible smooth representation in $ \MOD_G^{\sm}(\O/\varpi) $. Moreover, we have an isomorphism  \begin{equation}\label{decomp Pi into Pi^0}
	\Pi\cong (\plim_n\Pi^0/\varpi^n\Pi^0)[1/\varpi]. 
	\end{equation}
	
	Since the representation $ \Pi $ is a topological $ G $-module, we can consider its continuous group cohomology groups and we will show that the isomorphism \eqref{decomp Pi into Pi^0} induces an isomorphism in cohomology:
	\begin{equation} \label{Banach main result}
		H^i(G,\Pi)\cong (\plim_nH^i(G,\Pi^0/\varpi^n\Pi^0))[1/\varpi].
	\end{equation}
	
	The isomorphism \eqref{Banach main result} was proved for compact $ p $-adic analytic groups by Emerton in \cite[Proposition 1.2.19 and 1.2.20]{emerton2006interpolation}. 
	The main problem extending his proof is that, in general, projective limits do not commute with taking cohomology. Therefore we will first prove some finiteness properties on the cohomology groups.
	
	\subsection{Finiteness conditions}
	
	\begin{lem}[{\cite[Lemma 3.4.4]{Emerton}}] \label{compact implies finite}
		If $ K $ is a compact $ p $-adic analytic group, $ \pi $ an admissible representation in $ \MOD_K^{\sm}(\O/\varpi^n) $, $ n> 0 $, then the continuous group cohomology groups $ H^i(K,\pi) $ are finitely generated $ \O/\varpi^n $-modules for all $ i\ge 0 $.
	\end{lem}
	
	\begin{proof} 
		Let $ \pi $ be admissible representation on an $ \O/\varpi^n $-module. 
		Then, by Lemma 2.2.11 in \cite{EmertonEins}, its Pontryagin dual $ \pi^{\vee}=\Hom_\O^{\cts}(\pi,L/\O)$ is a finitely generated $ \OnK $-module (for more details see Corollary 1.8 and Proposition 1.9.(i) in \cite{Kohlhaase}) and since $ K $ is compact $ p $-adic analytic, the ring $ \OnK $ is Noetherian, by \cite[Theorem 2.1.1]{EmertonEins}. Hence, we can find a resolution $ F_\bullet \twoheadrightarrow \pi^{\vee} $ by free $ \OnK $-modules of finite rank. Then, taking the dual, we will get an injective resolution $ \pi \hookrightarrow {(F_\bullet)}^{\vee}$ of $ \pi $ in $ \MOD^{\sm}_K(\O/\varpi^n) $ 
		and the $ K $-invariants $ ({{F_i}^{\vee}})^K $ are finitely generated as $ \O/\varpi^n $-modules. Then, since $ \OnK $ is Noetherian, using Proposition \ref{compact open} we deduce that \[ H^i(K,\pi)\cong \Ext^i_K(\1,\pi)=H^i(({(F_\bullet)^{\vee}})^K) \] is finitely generated over $ \O /\varpi^n$ for all $ i $.
	\end{proof}
		
	\begin{lem} \label{contains compact implies finite}
		Let $ H $ be a locally profinite group and let $ \pi\in\MOD_H^{\adm}(\O/\varpi^n) $ be an admissible representation of $ H $. Assume that $ H $ contains an open normal subgroup $ N\le H $ such that \begin{itemize}
			\item $ H^i(N,\pi) $ is a finitely generated $ \O/\varpi^n $-module for all $ i\ge 0 $;
			\item the quotient group $ H/N $ is either finite or a finitely generated abelian group.
		\end{itemize}
		Then the $ \O/\varpi^n $-modules $ H^i(H,\pi) $ are finitely generated for all $ i\ge 0 $.
	\end{lem}

	\begin{proof}
		Again, by our main result \ref{main result}, we can compute the continuous group cohomology as the Ext group in the category $ \MOD^{\sm}_H(\O/\varpi^n)  $ of smooth representations over $ \O/\varpi^n $. Since the functor of taking $ H $-invariants is the composition of the functors \[ \MOD^{\sm}_H(\O/\varpi^n) \overset{(-)^{N}}{\longrightarrow}  \MOD^{\sm}_{H/N}(\O/\varpi^n) \overset{(-)^{H/N}}{\longrightarrow}  \MOD(\O/\varpi^n), \] we obtain a Hochschild-Serre spectral sequence \[ E^{i,j}_2=\Ext^i_{H/N}(\1,\Ext^j_N(\1,\pi))\implies \Ext^{i+j}_H(\1,\pi). \]
		Hence, we obtain the same for the continuous cohomology groups:
		\[ E^{i,j}_2=H^i(H/N,H^j(N,\pi))\implies H^{i+j}(H,\pi). \]
		Since this is a first quadrant spectral sequence, for fixed $ (i,j) $, the limit term is $ E^{i,j}_\infty=E^{i,j}_m$ for some $ m $. Moreover, all of the modules $ H^j(N,\pi) $ are finitely generated and since the quotient group $ H/N $ is discrete, $ H^i(H/N,H^j(N,\pi)) $ is the group cohomology of $ H^j(N,\pi) $.
		By assumption, the quotient group $ H/N $ is either finite or finitely generated and abelian. In both cases, the group ring $ \O/\varpi^n[H/N] $ is Noetherian and one can compute the cohomology groups $ H^i(H/N,H^j(N,\pi)) $ using resolutions consisting of finitely generated $ \O/\varpi^n[H/N] $-modules. Therefore, all of the terms $ E^{i,j}_m $ are finitely generated as $ \O/\varpi^n[H/N] $-modules and, since $ H/N $ acts trivially on each $ E^{i,j}_m $, they are also finitely generated as $ \O/\varpi^n $-modules. By definition of the convergence of a spectral sequence, there is a finite filtration of $ H^i(H,\pi) $ such that the graded pieces are isomorphic to one of these finitely generated modules $ E^{i,j}_\infty $ and hence, $ H^i(H,\pi) $ itself is also a finitely generated $ \O/\varpi^n $-module.
	\end{proof}
	
	\subsection{$ p $-adic reductive groups}
   
      	Let $ F $ be a finite field extension of $ \Q_p $, let $ \G $ be a connected reductive group over $ F $ and $ G=\G(F) $ its group of $ F $-rational points. Following the notations of \cite{schneider1997representation}, we denote by $ X $ the reduced Bruhat--Tits building associated to $ G $ and for any $ q\ge 0 $ we denote by $ X_q $ the set of all $ q $-dimensional facets in $ X $ and by \[ X^q:=\bigcup_{F\in X_q}\overline{F} \]the $ q $-skeleton. Then for any facet $ F\in X_q $, the relative homology group $ H_q(X^q,X^q\backslash F; \Z) $ is a free $ \Z $-module of rank 1 and we define the set of oriented $ q $-facets to be \[ X_{(q)}:=\{ (F,c) \vert F\in X_q,\ c\in H_q(X^q,X^q\backslash F; \Z)\text{ a generator} \}. \]
      	
      	Moreover, for a $ q $-facet $ F\in X_{q} $, we define $ P_F^\dagger $ to be the $ G $-stabilizer of $ F $. This group then acts on the homology group $ H_q(X^q,X^q\backslash F; \Z) \cong \Z$ and therefore, it acts via a character \[ \delta_F:P_F^\dagger\rightarrow \{ \pm 1 \}. \]
      	
      	Following Section II.1 in \cite{schneider1997representation}, we obtain a $ G $-equivariant resolution of $ \Z $\[ \dots \rightarrow H_1(X^1,X^0; \Z)\rightarrow H_0(X^0; \Z)=\bigoplus_{F\in X_0}\Z\rightarrow \Z \rightarrow 0, \] 
      	since it computes the singular cohomology of $ X $, which is trivial because $ X $ is contractible. Moreover the authors define $ G $-equivariant isomorphisms \[ C^{\oriented}_c(X_{(q)},\Z)\overset{\cong}{\rightarrow} H_q(X^q,X^{q-1}; \Z) ,\] where\begin{align*}
      	 C^{\oriented}_c(X_{(q)},\Z)=\left\{
     \omega:X_{(q)}\rightarrow \Z \ \middle\vert \begin{array}{l}  \omega \text{ has finite support}\\ \omega((F,-c))=-\omega((F,c))\text{ for any }(F,c)\in X_{(q)} \end{array}\right \}.
      	\end{align*}
      	
      	We fix for any $ q\ge 0 $, a set $ R_q$ containing exactly one representative from every $ G $-orbit in $ X_q $.
      	Let us now fix a $ q\ge 0 $ and write $ R_q=\{F_1,\dots,F_{s_q}\} $ and fix for any $ F_i$ an orientation $ c_i $, so that $ (F_i,c_i)\in X_{(q)} $.
      	\begin{lem}
      		For every $ q\ge 0 $, there are $ G $-equivariant isomorphisms \begin{align*}
	      		C^{\oriented}_c(X_{(q)},\Z)&\overset{\cong}{\longrightarrow} \bigoplus_{i=1}^{s_q} \cind_{P_{F_i}^\dagger}^G(\delta_{F_i}) \\
	      		\omega &\mapsto (\omega_i)_i,
      		\end{align*}
      		where $ \omega_i:G\rightarrow \Z $ is given by $ g\mapsto \omega((g^{-1}F_i,g^{-1}c_i)) $.
      	\end{lem}
      	
      	\begin{proof}
      		Since $ \omega $ has finite support, the support of each $ \omega_i $ is compact modulo the stabilizer $ P_{F_i}^\dagger $ and by definition, \[ \omega_i(hg)=\omega((g^{-1}h^{-1}F_i,g^{-1}h^{-1}c_i))=\omega((g^{-1}F_i,\delta_{F_i}(h^{-1})g^{-1}c_i))=\delta_{F_i}(h)\omega_i(g) ,\] for all $ h\in P_{F_i}^\dagger $, $ g\in G $. In particular, $ \omega_i $ defines an element of $ \cind_{P_{F_i}^\dagger}^G(\delta_{F_i}) $.
      		
      		The map is therefore well-defined and it is clearly injective. Moreover, it is $ G $-equivariant. Indeed, take any $ g,h\in G $, then  \[ (g\omega)_i(h)=(g\omega)((h^{-1}F_i,h^{-1}c_i))=\omega((g^{-1}h^{-1}F_i,g^{-1}h^{-1}c_i))=\omega_i(hg)=(g\omega_i)(h).  \]
      		
      		To prove surjectivity, it is enough to show that the image contains a set of generators of the $ G $-module $ \bigoplus_{i=1}^{s_q} \cind_{P_{F_i}^\dagger}^G(\delta_{F_i}) $. We consider, for every $ i\in\{1,\dots, s_q \} $, the map $ \omega^{(i)}\in  C^{\oriented}_c(X_{(q)},\Z)$, which is supported on the set $ \{(F_i,c_i),(F_i,-c_i)\} $ and maps $ (F_i,c_i) $ to $ 1 $. Then $ (\omega^{(i)})_j =0$ for $ i\ne j $ and $  (\omega^{(i)})_i\in \cind_{P_{F_i}^\dagger}^G(\delta_{F_i}) $ is supported on $ P_{F_i}^\dagger $ and $ (\omega^{(i)})_i(1)=1 $. These elements clearly generate $ \bigoplus_{i=1}^{s_q} \cind_{P_{F_i}^\dagger}^G(\delta_{F_i}) $ as $ G $-module, showing surjectivity.

      	\end{proof}
      	
      	We get therefore an exact resolution of $ \Z $ in $ \MOD^{\sm}_G(\Z) $: 
      	\begin{equation}\label{resolution of 1}
      	\dots\rightarrow\bigoplus_{F\in R_1}\cind^G_{P_F^\dagger}\delta_F \rightarrow\bigoplus_{F\in R_0}\cind^G_{P_F^\dagger}\delta_F\rightarrow \Z \rightarrow 0.
      	\end{equation}
      	
      	Note that, in general, the stabilizers $ P_F^\dagger $ are not compact. However, in \cite[Section 1.2]{schneider1997representation}, the authors construct a compact open normal subgroup $ R_F$ of $ P_F^\dagger$. Moreover, they show that $ ZR_F $ is of finite index in $ P_F^\dagger $, where $ Z $ denotes the $ F $-rational points of the connected center of $ \G $. 
      	
	\begin{lem}\label{reductive implis finite}
		Let $ G $ be a $ p $-adic reductive group, $ \pi\in \MOD^{\adm}_G(\O/\varpi^n) $, then $ H^i(G,\pi) $ is a finitely generated $ \O/\varpi^n  $-module.
	\end{lem}

	\begin{proof}
		Since $ \pi $ is a smooth representation, we can apply Corollary \ref{main result}, to get \[ H^i(G,\pi)\cong \Ext^i_G(\1,\pi), \] where the right hand side is computed in the category of smooth representations of $ G $ on $ \O/\varpi^n $-modules. Since all terms in the resolution \eqref{resolution of 1} are free abelian groups, by the Künneth formula (\cite[Theorem 3.6.1]{weibel1995introduction}), it will remain exact after tensoring with $ \O/\varpi^n $ and we get a resolution of the trivial representation in $ \MOD^{\sm}_G(\O/\varpi^n) $: \[\dots\rightarrow\bigoplus_{F\in R_1}\cind^G_{P_F^\dagger}\delta_F \rightarrow\bigoplus_{F\in R_0}\cind^G_{P_F^\dagger}\delta_F\rightarrow \1 \rightarrow 0.  \]
		
		Let $ \pi\hookrightarrow I^\bullet $ be an injective resolution of $ \pi $ in $\MOD^{\sm}_G(\O/\varpi^n) $ and consider the double complex  \[ C^{i,j}:=\Hom_G(\bigoplus_{F\in R_i}\cind^G_{P_F^\dagger}\delta_F,I^j) . \]
		
		Since the objects $ I^j $ are injective, the complexes \[  0\rightarrow\Hom_G(\1,I^j)\rightarrow\Hom_G(\bigoplus_{F\in R_\bullet}\cind^G_{P_F^\dagger}\delta_F,I^j)\]
		are exact. Therefore, the spectral sequence associated to the horizontal filtration of $ C^{\bullet,\bullet} $ is given by  \[ E_{2,h}^{i,j}=\begin{cases}
		\Ext^j_G(\1,\pi),&\text{if } i=0,\\ 0, &\text{if } i>0.
		\end{cases} \]
		
		On the other hand, since the cohomology of the columns is \[  H^j(\Hom_G(\bigoplus_{F\in R_i}\cind^G_{P_F^\dagger}\delta_F,I^\bullet)) =\Ext^j_G(\bigoplus_{F\in R_i}\cind^G_{P_F^\dagger}\delta_F,\pi)\\ \cong \bigoplus_{F\in R_i}\Ext^j_{P_F^\dagger}(\delta_F,\pi) , \]using that $ \Ext^j_G(\cind^G_{P_F^\dagger}\delta_F,\pi)\cong \Ext^j_{P_F^\dagger}(\delta_F,\pi)$, by Lemma \ref{Vigneras Lemma on Ext groups}. Hence, the vertical filtration will give a spectral sequence \[ E_{2,v}^{i,j}=H^i(\bigoplus_{F\in R_\bullet}\Ext^j_{P_F^\dagger}(\delta_F,\pi)) \implies \Ext^{i+j}_G(\1,\pi).\] 
		
		Therefore, the claim follows from the fact that $\Ext^i_{P_F^\dagger}(\delta_F,\pi)\cong \Ext^i_{P_F^\dagger}(\1,\pi\otimes\delta_F^{-1})  $ is a finitely generated $ \O/\varpi^n $-module for every $ i\ge 0 $. Indeed, by Lemma \ref{compact implies finite}, the groups $ H^i(R_F,\pi\otimes\delta_F^{-1}) $ are finitely generated over $ \O/\varpi^n $ for all $ i\ge 0 $. Since the quotient $ ZR_F/R_F $ is a finitely generated abelian group, we can apply Lemma \ref{contains compact implies finite} to this pair of groups $ R_F\le ZR_F $, to get that $ H^i(ZR_F,\pi\otimes\delta_F^{-1}) $ is finitely generated and the Lemma \ref{contains compact implies finite} again applied to the pair $ ZR_F\le P_F^\dagger $ implies that $ H^i(P_F^\dagger,\pi\otimes\delta_F^{-1}) $ is finitely generated and the claim follows by Corollary \ref{main result}.
			\end{proof}
		
	\begin{lem}\label{lemma invariants direct limit}
		Let $ \{\pi_i\}_{i\in \mathbb{N}} $ be a directed system of $ G $-representations $ \pi_i\in \MOD_G^{\sm}(\O) $, such that the transition maps are injective. Then the natural map \[ \Lim_i (\pi_i)^G\overset{\sim}{\longrightarrow} (\Lim_i\pi_i)^G \] is an isomorphism.
	\end{lem}

	\begin{proof}
		It is easy to see that the map is injective. To show surjectivity, note that two elements $ v_1, v_2\in \pi_{j} $ are equivalent in the limit $ \Lim_i \pi_i $ if and only if they are equal in $ \pi_j $. Indeed, by definition they are equivalent if and only if their images in some $ \pi_n $ are equal for some $ n\ge j $, but all transition maps are injective and hence this is equivalent to being equal in $ \pi_j $. In particular, if the image of an element $ v\in \pi_j $ in the limit is $ G $-invariant, then it is already $ G $-invariant in $ \pi_j $.
	\end{proof}

	For an $ \O $-module $ M $, we say that it is $ \varpi^\infty $-torsion, if every element in $ M $ is annihilated by a power of $ \varpi $.

	\begin{lem}\label{H(Pi/Pi^0)is torsion}
		Let $ G $ be a $ p $-adic reductive group, $ \Pi $ an admissible unitary Banach space representation of $ G $ over $ L $ and $ \Pi^0 $ an open bounded $ G $-invariant lattice in $ \Pi $. Then for all $ i\ge 0 $, we have \[ H^i(G,\Pi/\Pi^0)\cong \Lim_nH^i(G,\varpi^{-n}\Pi^0/\Pi^0). \]
		In particular, $H^i(G,\Pi/\Pi^0)$ is $ \varpi^\infty $-torsion, i.e. $ H^i(G,\Pi/\Pi^0)[1/\varpi]=0  $.
	\end{lem}

	\begin{proof}
		The representation $ \Pi/\Pi^0 $ is discrete and $ \varpi^\infty $-torsion. Moreover, it lies in the category $ \MOD^{\sm}_G(\O) $ and is isomorphic to the direct limit \[ \Pi/\Pi^0 \cong \Lim_n \varpi^{-n}\Pi^0/\Pi^0 . \] 
		If $ K\le G$ is a compact open subgroup, the terms in the complex of continuous cochains $ C^\bullet(K,\Pi/\Pi^0) $ commute with the direct limit, since the image of any map in $ C^n(K,\Pi/\Pi^0)\cong C^n(K,\Lim_n\varpi^{-n}\Pi^0/\Pi^0) $ will be contained in the image of finitely many $ \varpi^{-n}\Pi^0/\Pi^0 $ in the limit. (Note that this argumentation does not apply to a non-compact $ G $.) Moreover, by Lemma \ref{lemma invariants direct limit}, we get \[ C^n(K,\Pi/\Pi^0)^K\cong \Lim_nC^n(K,\varpi^{-n}\Pi^0/\Pi^0)^K, \]and since taking direct limits is exact in the category of $ \O $-modules (cf. \cite[Theorem 2.6.15]{weibel1995introduction}), the cohomology groups $ H^i(K,-) $ commute with the direct limit, i.e. \[ H^i(K,\Pi/\Pi^0)\cong \Lim_nH^i(K,\varpi^{-n}\Pi^0/\Pi^0). \]  
		
		To prove the statement for a non-compact $ G $, we proceed as in the proofs of the Lemmas \ref{contains compact implies finite} and \ref{reductive implis finite}, to reduce the proof to the compact situation. To adapt the proofs, note that if $ H $ is a discrete group, the continuous group cohomology of $ H $ is just group cohomology and if we assume that $ H $ is either finite or finitely generated and abelian, then the group ring $ \mathbb{Z}[H] $ is Noetherian and therefore, $ H^i(H,-)\cong \Ext^i_{\Z[H]}(\Z,-) $ can be computed using a projective resolution of $ \Z $ consisting of free $ \Z[H] $-modules of finite rank, hence $ H^i(H,-) $ commutes with direct limits. Then we argue as in the proof of Lemma \ref{reductive implis finite} to obtain a spectral sequence \[ E_{2}^{i,j}=H^i(\bigoplus_{F\in R_\bullet}\Ext^j_{P_F^\dagger}(\delta_F,\Pi/\Pi^0)) \implies \Ext^{i+j}_G(\1,\Pi/\Pi^0).\] 
		Moreover, by the Comparison Theorem 5.2.12 in \cite{weibel1995introduction}, it is then enough to show that the terms in the spectral sequence commute with the direct limit, i.e. it is enough to show that \[ \Ext^i_{P_F^\dagger}(\delta_F,\Pi/\Pi^0)\cong \Lim_n \Ext^i_{P_F^\dagger}(\delta_F,\varpi^{-n}\Pi^0/\Pi^0),\]
		for every $ q\ge 0 $ and $ F\in R_q $. Or equivalently, \[ H^i(P_F^\dagger,\Pi/\Pi^0\otimes \delta_F^{-1})\cong \Lim_n H^i(P_F^\dagger,\varpi^{-n}\Pi^0/\Pi^0\otimes \delta_F^{-1}).\] This follows then from the fact that this is true for the compact open subgroup $ R_F $, using as in the proof of Lemma \ref{contains compact implies finite} the Hochschild-Serre spectral sequences associated to $ R_F\le ZR_F $, $ ZR_F\le P_F^\dagger $: \[ E_2^{i,j}=H^i(ZR_F/R_F,H^j(R_F,\pi)) \implies H^{i+j}(ZR_F,\pi) \]
		\[ E_2^{i,j}=H^i(P^\dagger_F/ZR_F,H^j(ZR_F,\pi)) \implies H^{i+j}(P^\dagger_F,\pi) \]
		The $ E_2 $-terms all commute with direct limits by the above discussion, hence so do the targets.
		
		In particular, since each of the $ H^i(G,\varpi^{-n}\Pi^0/\Pi^0) $ is $ \varpi^\infty $-torsion, so is the direct limit $ \Lim_nH^i(G,\varpi^{-n}\Pi^0/\Pi^0)\cong H^i(G,\Pi/\Pi^0) $.
	\end{proof}
	
	\begin{prop} \label{H=limH} 
		Let $ G $ be a $ p $-adic reductive group, $ \Pi $ an admissible unitary Banach space representation of $ G $ over $ L $ and $ \Pi^0 $ an open bounded $ G $-invariant lattice in $ \Pi $. Then we have \[ H^i(G,\Pi^0)\cong \plim_n H^i(G,\Pi^0/\varpi^n\Pi^0), \forall i\ge 0. \]
		In particular, \[ H^i(G,\Pi)\cong (\plim_n H^i(G,\Pi^0/\varpi^n\Pi^0))[1/\varpi]. \]
	\end{prop}

	\begin{proof}
		We consider the tower of cochain complexes $ \dots \rightarrow C_2^\bullet\rightarrow C_1^\bullet\rightarrow C_0^\bullet$, where $ C_n^i:=C^i(G,\Pi^0/\varpi^n\Pi^0)^G $ is the cochain complex computing the continuous group cohomology of $ \Pi^0/\varpi^n\Pi^0 $. This projective system of cochain complexes satisfies the Mittag-Leffler condition, because each of the maps $ C^i_n\rightarrow C^i_m $, $ n\ge m $, is surjective. Indeed, since $ \Pi^0/\varpi^m\Pi^0 $ is discrete, the short exact sequence of representations  \[ 0\rightarrow \varpi^m\Pi^0/\varpi^n\Pi^0\rightarrow\Pi^0/\varpi^n\Pi^0\rightarrow\Pi^0/\varpi^m\Pi^0 \rightarrow 0 \] induces a short exact sequence of complexes \[ 0\rightarrow C^\bullet(G,\varpi^m\Pi^0/\varpi^n\Pi^0)\rightarrow C^\bullet(G,\Pi^0/\varpi^n\Pi^0)\rightarrow C^\bullet(G,\Pi^0/\varpi^m\Pi^0) \rightarrow 0 , \]
		which stays exact after taking $ G $-invariants, since the module $ C^i(G,\varpi^m\Pi^0/\varpi^n\Pi^0) $ is acyclic for the continuous group cohomology for all $ i $.
		
		We can therefore apply Theorem 3.5.8 of \cite{weibel1995introduction}, to get a short exact sequence \[ 0\rightarrow {\plim_n}^{(1)} H^{i-1}(C_n^\bullet)\rightarrow H^i(\plim_n C_n^\bullet)\rightarrow \plim_n H^i(C_n^\bullet) \rightarrow 0,\]
		where $ \plim_n^{(1)} $ denotes the right derived functor of the projective limit functor. But by Theorem 1 in \cite{jensen1970vanishing}, the functor $ \plim_n^{(1)} $ applied to a projective system of finitely generated modules over a complete Noetherian local ring, such as $ \O $, vanishes. Hence, by Lemma \ref{reductive implis finite}, \[  {\plim_n}^{(1)} H^{i-1}(C_n^\bullet)={\plim_n}^{(1)} H^{i-1}(G,\Pi^0/\varpi^n\Pi^0)=0  \] and we get an isomorphism \[ H^i(\plim_n C_n^\bullet)\overset{\cong}{\longrightarrow}\plim_n H^i(C_n^\bullet)=\plim_n H^i(G,\Pi^0/\varpi^n\Pi^0). \]
		Since the inverse limit $ \plim_n\Pi^0/\varpi^n\Pi^0 $ can be seen as the inverse limit in the category of topological spaces, one has isomorphisms  \[ C(G^{i+1},\plim_n\Pi^0/\varpi^n\Pi^0)\cong \plim_nC(G^{i+1},\Pi^0/\varpi^n\Pi^0),\]
		inducing isomorphisms on the $ G $-invariants  \[ C^i(G,\plim_n\Pi^0/\varpi^n\Pi^0)^G \cong (\plim_nC^i(G,\Pi^0/\varpi^n\Pi^0))^G\cong\plim_n  C^i(G,\Pi^0/\varpi^n\Pi^0)^G.\]
		The second isomorphism follows again from the universal property of projective limits, since $ \Hom_G(\1,\plim_n C^i(G,\Pi^0/\varpi^n\Pi^0))\cong \plim_n \Hom_G(\1, C^i(G,\Pi^0/\varpi^n\Pi^0)) $. Therefore, the cohomology of $ \plim_nC_n^\bullet $ is in fact $ H^i(G,\plim_n \Pi^0/\varpi^n\Pi^0) =H^i(G,\Pi^0)$.
		
		It remains to show that $ H^i(G,\Pi)\cong H^i(G,\Pi^0)[1/\varpi] $. For this, consider the short exact sequence of $ G $-modules \[ 0\rightarrow \Pi^0 \rightarrow \Pi \rightarrow \Pi/\Pi^0\rightarrow 0 .\] 
		
		The quotient $ \Pi/\Pi^0 $ is discrete, therefore we get a long exact sequence in cohomology \[ \dots \rightarrow H^{i-1}(G,\Pi/\Pi^0)\rightarrow H^i(G,\Pi^0)\rightarrow H^i(G,\Pi)\rightarrow H^i(G,\Pi/\Pi^0)\rightarrow \dots  .\]
		
		Since localization is exact, $ H^i(G,\Pi/\Pi^0) $ is $ \varpi^\infty $-torsion by Lemma \ref{H(Pi/Pi^0)is torsion} and $ H^i(G,\Pi) $ is an $ L $-vector space, we get \[ H^i(G,\Pi^0)[1/\varpi]\cong H^i(G,\Pi)[1/\varpi]\cong H^i(G,\Pi). \]
	\end{proof}

	\begin{cor} \label{H^i(G,Pi) are finite diml}
		In the notation of Proposition \ref{H=limH}, the $ L $-vector spaces $ H^i(G,\Pi) $ are finite dimensional, for all $ i\ge 0 $.
	\end{cor}

	\begin{proof}
		Since $ H^i(G,\Pi)\cong H^i(G,\Pi^0)[1/\varpi] $, it suffices to show that $ H^i(G,\Pi^0) $ is a finitely generated $ \O $-module. We have a short exact sequence \[ 0\longrightarrow \Pi^0\overset{\varpi}{\longrightarrow}\Pi^0\longrightarrow \Pi^0/\varpi\Pi^0\longrightarrow 0, \]
		which induces the long exact sequence of $ \O $-modules \[ \dots\rightarrow H^i(G,\Pi^0)\overset{\varpi}{\longrightarrow}H^i(G,\Pi^0)\longrightarrow H^i(G,\Pi^0/\varpi\Pi^0)\longrightarrow \dots .\]
		In particular, the quotient $ H^i(G,\Pi^0)/\varpi H^i(G,\Pi^0)$ can be embedded into the $ \O $-module $ H^i(G,\Pi^0/\varpi\Pi^0) $, which is finite by Lemma \ref{reductive implis finite}. Moreover, $ H^i(G,\Pi^0) $ is profinite by Lemma \ref{reductive implis finite} and Proposition \ref{H=limH}, hence it is a compact $ \O $-module and the claim follows from the topological Nakayama's lemma (cf. \cite[§3, Corollary]{balister1997note}).
	\end{proof}

	\subsection{Banach space representations of $ \GL_2(\Q_p) $} \label{sec:banach of GL2}
	
	We use Proposition \ref{H=limH} to prove an analogue of Corollary \ref{H(SL2,pi)=0} for Banach space representations. 
	
	From now on, let $ G=\GL_2(\Q_p)$, let $ \zeta:\Q_p^\times\rightarrow  L^\times$ be a unitary character and let $ \Badm_{G,\zeta}(L) $ be the full subcategory of $ \Badm_G(L) $ consisting of objects which have central character $ \zeta $. This category does not have enough injectives or projectives, but we can consider the Yoneda Ext-groups in this category $ \Ext_{\Badm_{G,\zeta}(L)} ^i(\Pi_1,\Pi_2)$. 
	
	We want to prove the following
	\begin{prop} \label{prop Ext=limExt}
		For $ \Pi_1, \ \Pi_2\in \Badm_{G,\zeta}(L) $, with $ \Pi_1 $ of finite length, one has  \begin{equation}\label{Ext=limExt}
			\Ext^i_{\Badm_{G,\zeta}(L)} (\Pi_1,\Pi_2)\cong \left(  \plim_n\Ext^i_{\MOD_{G,\bar{\zeta}}^{\sm}(\O/\varpi^n)}(\Pi_1^0/p^n,\Pi_2^0/p^n)\right) [1/p], \forall i\ge 0,
			\end{equation} 
			where $ \bar{\zeta} $ is the composition of $ \zeta $ with the projection map $ \O\twoheadrightarrow \O/\varpi^n $.
	\end{prop}

	By Proposition 5.34 and Proposition 5.36 of \cite{paskColmez}, we know that both categories $ \Badm_{G,\zeta}(L) $ and $ \MOD_{G,\zeta}^{\lfin}(\O) $ decompose into the direct sum, resp. product, of subcategories \begin{align*}
		\Badm_{G,\zeta}(L)\cong \bigoplus_\mathcal{B}\Badm_{G,\zeta}(L)_{\mathcal{B}}, \\
		\MOD_{G,\zeta}^{\lfin}(\O) \cong \prod_\mathcal{B}\MOD_{G,\zeta}^{\lfin}(\O) _{\mathcal{B}}.
	\end{align*}
	Here, the blocks $ \mathcal{B} $ are the same equivalence classes of irreducible representations in $ \MOD_{G,\zeta}^{\sm}(k) $. A representation $ \Pi \in \Badm_{G,\zeta}(L)  $ lies in a subcategory $ \Badm_{G,\zeta}(L)_{\mathcal{B}} $, if for an open bounded $ G $-invariant lattice $ \Pi^0\subset\Pi $, all irreducible subquotients of the smooth $ k $-representation $ \Pi^0\otimes_\O k $ lie in $ \mathcal{B} $.
	
	In particular, there are no extensions between representations lying in different blocks and it is enough to show \eqref{Ext=limExt} for each block individually. Fix such block $ \mathcal{B} $ and consider the category $ \MOD_{G,\zeta}^{\lfin}(\O)_{\B} $ and let $ \mathfrak{C}(\O) $ be the strictly full subcategory of $ \MOD_G^{\text{pro aug}}(\O) $ which is anti-equivalent to $  \MOD_{G,\zeta}^{\lfin}(\O)_{\B}$ via Pontryagin duality. Corollary 1.2 in \cite{pavskunas2014blocks} gives a full description of the possible blocks $ \B $ and shows in particular, that all of the blocks contain only finitely many elements. Let $ \pi_1 ,\dots, \pi_s$ be representatives of the isomorphism classes of irreducible $ k $-representations in $ \mathcal{B} $. Then the representation $   \oplus_{i=1}^s \pi_i^\vee$ lies in $ \mathfrak{C}(\O) $ and has a projective envelope $ \textP $. As in Section 2 of \cite{paskColmez}, we define the ring $\E=\End_{\mathfrak{C}(\O)}(\textP) $ as the endomorphism ring of this projective envelope. This endomorphism ring is equipped with its so-called natural topology (cf. \cite[p.14]{paskColmez}), making it a pseudo-compact ring. In fact, the natural topology is the same as the $ \mathfrak{m} $-adic topology, where $ \mathfrak{m} $ is the maximal two-sided ideal of $ \E $ (\cite[Corollary 3.11]{paskColmez}). We obtain an equivalence of categories (cf. Section 4.2 in \cite{paskColmez}) \begin{align*}
	 \mathfrak{C}(\O) &\overset{\sim}{\longrightarrow} \{\text{compact (right) }\E\text{-modules} \},  \\ M&\mapsto \Hom_{\mathfrak{C}(\O)}(\textP,M) 
	\end{align*} with inverse given by the completed tensor product $ \m\mapsto \m\widehat{\otimes}_{\E} \textP $, for a compact $ \E $-module $ \m $.
	
	The ring $ \E $ is finitely generated over its center, which is Noetherian (see Section 6.4 of \cite{pavskunas2021finiteness}), and hence our setup satisfies the assumptions of Section 4.2 in \cite{paskColmez}. As constructed in Lemma 4.9 of \cite{paskColmez}, we have a fully faithful contravariant functor  \[ \m:\Banb\rightarrow \MOD^{\fg}_{\E[1/p]} , \] where $ \MOD^{\fg}_{\E[1/p]} $ is the category of finitely generated $ \E[1/p] $-modules. The functor $ \m $ is defined as follows: Let $ \Pi \in\Banb$ and let $ \Pi^0 $ be an open bounded $ G $-invariant lattice of $ \Pi $. Then its Schikhof dual $ (\Pi^0)^d:=\Hom_\O(\Pi^0,\O) $ is an element of $ \mathfrak{C}(\O) $, by Lemma 4.11 in \cite{paskColmez}, so by applying $ \Hom_{\mathfrak{C}(\O)}(\textP,-) $, we obtain a compact $ \E $-module. We set \[ \m(\Pi):=\Hom_{\mathfrak{C}(\O )}(\textP,(\Pi^0)^d )\otimes_\O L. \] It does not depend on the choice of the lattice $ \Pi^0 $. This functor $ \m $ induces an isomorphism  \[ \Ext^i_{\Banb}(\Pi_1,\Pi_2)\cong \Ext^i_{\E[1/p]}(\m(\Pi_2),\m(\Pi_1)),\]
	for all $ i\ge 0 $ and for $ \Pi_1, \ \Pi_2\in \Banb$, with $ \Pi_1 $ of finite length (\cite[Corollary 4.48]{paskColmez}). The Ext-group on the right hand side can then be expressed as a projective limit in the following way:
	
	\begin{lem}\label{Ext=lim for Ep-modules}
		Let $ \m_1 $, $ \m_2 $ be finitely generated $ \Ep $-modules. Then there exist $ \E $-stable $ \O $-lattices $ \m_1^0 $, $ \m_2^0 $ in $ \m_1 $ and $ \m_2 $ respectively, which are finitely generated as $ \E $-modules and we have \[ \Ext^i_{\Ep}(\m_2,\m_1)\cong \left(\plim_n \Ext^i_{\E/p^n}(\m_2^0/p^n,\m_1^0/p^n)\right)[1/p], \]
		where the groups $ \Ext_{\E[1/p]}^i$ and $ \Ext^i_{\E/p^n} $ are computed in the category of finitely generated $ \E[1/p] $- and $ \E/p^n $-modules, respectively.
	\end{lem}

	\begin{proof}
		We can choose a finite set of generators of the $ \Ep $-module $ \m_i $ and let $ \m_i^0 $ be the $ \E $-submodule of $ \m_i $ generated by those generators. Then $ \m_i^0 $ is an $ \O $-lattice in $ \m_i $. The ring $ \E $ is in fact pseudo-compact (by Proposition 13 in \cite[§IV.4]{gabriel1962categories}) and is Noetherian, as it is finitely generated over its center, which is Noetherian (\cite[Corollary 6.4]{pavskunas2021finiteness}). Therefore, the finitely generated $ \E $-modules $ \m_i^0 $ can be presented as the cokernels of morphisms $ \E^{\oplus n}\rightarrow \E^{\oplus m} $ of pseudo-compact $ \E $-modules and hence $ \m_i^0 $ equipped with the quotient topology, are also pseudo-compact (using Theorem 3 in chapter IV of \cite{gabriel1962categories}). Moreover, by Proposition 13 in Chapter IV of \cite{gabriel1962categories}, the quotient of $ \E $ by its Jacobson radical is isomorphic to $ \End_{\mathfrak{C}(\O)}(\oplus_{i=1}^n \pi_i^\vee) \cong \prod_{i=1}^{n}\End_{\mathfrak{C}(\O)}(\pi_i^\vee)$, which is a finite-dimensional vector space over the residue field of $ \O $. Hence,  the $ \O $-modules $ \m_i^0 $ are also pseudo-compact, thus they can be written as a projective limit $ \m_i^0\cong \plim_n\m_i^0/p^n $. We obtain \[ \m_i\cong (\plim_n \m_i^0/p^n)\otimes_\O L. \]
		Since the ring $ \E $ is Noetherian, we can compute the $ \Ext $-group $ \Ext^i_{\E}(\m_2^0,\m_1^0) $ using a projective resolution \[ \dots \rightarrow \E^{\oplus s_2}\rightarrow \E^{\oplus s_1}\rightarrow \E^{\oplus s_0}\rightarrow \m_2^0\rightarrow 0,\] by free $ \E $-modules of finite rank. Then \[ \Ext^i_{\E}(\m_2^0,\m_1^0)= H^i(\Hom_{\E}(\E^{\oplus s_\bullet},\plim_n \m_1^0/p^n))\cong H^i(\plim_n\Hom_{\E}(\E^{\oplus s_\bullet}, \m_1^0/p^n)) . \] But since the modules $ \Hom_{\E}(\E^{\oplus s_\bullet}, \m_1^0/p^n) $ are finitely generated, the projective limit commutes with taking the cohomology and we obtain \[ \Ext^i_{\E}(\m_2^0,\m_1^0)= \plim_n H^i(\Hom_{\E}(\E^{\oplus s_\bullet}, \m_1^0/p^n))=\plim_n \Ext^i_{\E}(\m_2^0,\m_1^0/p^n). \]
		Moreover, since the module $ \m_2^0 $ is $ \O $-torsion free, the resolution $ \E^{\oplus s_\bullet}\twoheadrightarrow \m_2^0 $ stays exact after tensoring with $ \E/p^n $. In particular, the resulting sequence $ (\E/p^n)^{\oplus s_\bullet}\twoheadrightarrow \m_2^0/p^n  $ is a projective resolution of $ \m_2^0/p^n $ as an $ \E/p^n $-module and we get \begin{align*} \Ext^i_{\E}(\m_2^0,\m_1^0/p^n)&=H^i(\Hom_{\E}(\E^{\oplus s_\bullet},\m_1^0/p^n))\\ &\cong H^i(\Hom_{\E/p^n}({(\E/p^n)}^{\oplus s_\bullet},\m_1^0/p^n))\\&=\Ext^i_{\E/p^n}(\m_2^0/p^n, \m_1^0/p^n). \end{align*}
	\end{proof}
	
	Note that, since $ \E $ is Noetherian, the groups $ \Ext_{\E/p^n }^i(\m_2^0/p^n,\m_1^0/p^n) $ computed in the category of finitely generated $ \E/p^n $-modules agree with the ones computed in the category of compact $ \E/p^n $-modules, which is anti-equivalent to $ \MOD_{G,\zeta}^{\lfin}(\O/\varpi^n)_{\B} $. Hence, we get the following lemma:
	
	\begin{lem}
		In the notation of the previous lemma, we have \[ \Ext^i_{\E/p^n}(\m_2^0/p^n,\m_1^0/p^n)\cong \Ext^i_{\MOD^{\sm}_{G,\bar{\zeta}}(\O/\varpi^n)}(((\m_1^0/p^n)\widehat{\otimes}_{\E}\textP)^{\vee},((\m_2^0/p^n)\widehat{\otimes}_{\E}\textP)^{\vee}) . \]
	\end{lem}
	
	\begin{proof}
		Via the anti-equivalence between the category of compact $ \E/p^n  $-modules and the category $  \MOD_{G,\bar{\zeta}}^{\lfin}(\O/\varpi^n)_{\B}  $, the modules $ \m_i^0/p^n $ correspond to $ ((\m_i^0/p^n)\widehat{\otimes}_{\E}\textP)^{\vee} $. The claim follows then from the decomposition of the category $ \MOD_{G,\bar{\zeta}}^{\lfin}(\O/\varpi^n) $ into blocks and the fact that the $ \Ext $-group of locally finite representations in $ \MOD_{G,\bar{\zeta}}^{\lfin}(\O/\varpi^n) $ is isomorphic to their $ \Ext $-group computed in the category of smooth representations $ \MOD^{\sm}_{G,\bar{\zeta}}(\O/\varpi^n) $ (see Corollary 5.17 and 5.18 in \cite{paskColmez}).
	\end{proof}
	
	\begin{lem}
	 Let $ \Pi_1 $ and $ \Pi_2 $ be admissible unitary Banach space representations in $ \Banb $ and set $ \m_i:=\m(\Pi_i) $. Then there are $ \E $-stable $ \O $-lattices $ \m_1^0, $ $ \m_2^0 $ in $ \m_1 $ and $ \m_2 $ respectively, such that \[ \Ext^i_{\MOD_{G,\bar{\zeta}}^{\sm}(\O/\varpi^n)}(\Pi_1^0/p^n,\Pi_2^0/p^n)\cong \Ext^i_{\MOD_{G,\bar{\zeta}}^{\sm}(\O/\varpi^n)}(((\m_1^0/p^n)\widehat{\otimes}_{\E}\textP)^{\vee},((\m_2^0/p^n)\widehat{\otimes}_{\E}\textP)^{\vee}), \] where $ \Pi_1^0 $ and $ \Pi_2^0 $ are open bounded $ G $-invariant lattices in $ \Pi_1 $ and $ \Pi_2 $, respectively.
	\end{lem}

	\begin{proof}
		Recall that for $ \Pi\in \Banb $, $\m(\Pi)$ was defined as $\Hom_{\mathfrak{C}(\O)}(\textP,(\Pi^0)^d)\otimes_\O L $. Hence, $ \m(\Pi)^0:=\Hom_{\mathfrak{C}(\O)}(\textP,(\Pi^0)^d) $ defines an $ \E $-stable $ \O $-lattice in $ \m(\Pi) $. To show that it satisfies the claimed isomorphism, we need to prove that the smooth representation $ \Pi^0/p^n\in\MOD_{G,\zeta}^{\lfin}(\O)_{\B}  $ is isomorphic to $ ((\m(\Pi)^0/p^n)\widehat{\otimes}_{\E}\textP)^\vee $. But by the anti-equivalence of categories between $  \MOD_{G,\zeta}^{\lfin}(\O)_{\B}  $ and compact $ \E  $-modules, this is equivalent to having an isomorphism \[ \Hom_{\mathfrak{C}(\O)}(\textP, (\Pi^0/p^n)^\vee)\cong \m(\Pi)^0/p^n, \] or in other words \[ \Hom_{\mathfrak{C}(\O)}(\textP, (\Pi^0/p^n)^\vee)\cong \Hom_{\mathfrak{C}(\O)}(\textP, (\Pi^0)^d)/p^n .\]
		We conclude by using the isomorphism (\cite[Section 2]{paskunas2018Ludwig})
		\[ (\Pi^0)^d/p^n\cong( \Pi^0/p^n)^\vee .\]
	\end{proof}

	These lemmas combined prove Proposition \ref{prop Ext=limExt}.\\

	To state the next proposition, we need to introduce the following notation:
	
	Let $ \1 $ be the trivial one-dimensional representation in $ \Badm_G(L) $, $ \widehat{\Sp} $ be the universal unitary completion of the smooth Steinberg representation of $ G $ over $ L $. Let $ B $ be the subgroup of upper triangular matrices in $ G $. Let $ \tilde{\alpha} :B\rightarrow L^\times $ be the representation of $ B $, defined by $ \begin{pmatrix}
	a&b\\0& d
	\end{pmatrix}\mapsto ad^{-1} \lvert ad^{-1}\rvert  $. Then $ \lvert \tilde{\alpha}(b) \rvert =1$ for all $ b\in B $ and we can define $ \Ind_B^G\tilde{\alpha}$ to be the induced representation, given by continuous functions $ f:G\rightarrow L $ with $ f(bg)=\tilde{\alpha}(b)f(g) $ for all $ b\in B $ and $ g\in G $, endowed with the supremum norm. This defines an admissible unitary Banach space representation of $ G $ (see Section 7.2 of \cite{paskColmez} for more details).

	\begin{prop}\label{H(SL2,Pi)=0}
		Let $ \Pi\in\Badm_{G,\zeta}(L) $ be an absolutely irreducible admissible unitary Banach space representation of $ G $, let $ \Pi^0 $ be an open bounded $ G $-invariant lattice in $ \Pi $ and assume that $ \Pi $ is not isomorphic to a twist by a unitary character of $ \1,\ \widehat{\Sp}$ or $ \Ind^G_B\tilde{\alpha}$. Then, for all $ i\ge 0 $, we have \[ H^i(\SL(\Q_p),\Pi) =0 .\]
	\end{prop}

	\begin{proof}
		We argue as in the proof of Corollary \ref{H(SL2,pi)=0}. Namely, by Proposition \ref{H=limH} and Corollary \ref{main result}, we have \begin{align*}
		H^i(\SL(\Q_p),\Pi)\cong (\plim_nH^i(\SL(\Q_p),\Pi^0/\varpi^n\Pi^0))[1/\varpi] \\ \cong  (\plim_n\Ext^i_{\MOD^{\sm}_{\SL(\Q_p)}(\O/\varpi^n)}(\1,\Pi^0/\varpi^n\Pi^0))[1/\varpi] .
		\end{align*}
		Let $ Z $ be the center of $ G $. Without loss of generality, we may assume that $ Z\cap\SL(\Q_p) $ acts trivially on $ \Pi $, since otherwise, Lemma \ref{wlog center acts trivially} implies that the groups $ H^i(\SL(\Q_p),\Pi) $ are zero for all $ i\ge 0 $. 
		Then we get an exact functor \[ (-)\vert_{\SL(\Q_p)}:\MOD^{\sm}_{Z\SL(\Q_p),\bar{\zeta}}(\O/\varpi^n)\rightarrow \MOD^{\sm}_{\SL(\Q_p),\1}(\O/\varpi^n) ,\] where $ \bar{\zeta} $ is the composition of $ \zeta $ with the projection onto $ \O/\varpi^n $. The functor $ (-)\vert_{\SL(\Q_p)} $ preserves injective objects and hence we obtain an isomorphism of $ \Ext $-groups \[\Ext^i_{\MOD^{\sm}_{Z\SL(\Q_p),\bar{\zeta}}(\O/\varpi^n)}(\bar{\zeta},\Pi^0/\varpi^n\Pi^0)\cong \Ext^i_{\MOD^{\sm}_{\SL(\Q_p),\1}(\O/\varpi^n)}(\1,\Pi^0/\varpi^n\Pi^0).  \]
		Using this, we get
		 \begin{align*}
		H^i(\SL(\Q_p),\Pi)\cong (\plim_n\Ext^i_{\MOD^{\sm}_{Z\SL(\Q_p),\bar{\zeta}}(\O/\varpi^n)}(\bar{\zeta},\Pi^0/\varpi^n\Pi^0))[1/\varpi] \\ \cong   (\plim_n\Ext^i_{\MOD^{\sm}_{G,\bar{\zeta}}(\O/\varpi^n)}(\Ind_{Z\SL(\Q_p)}^G\bar{\zeta},\Pi^0/\varpi^n\Pi^0))[1/\varpi] \\
		\cong (\plim_n\Ext^i_{\MOD^{\sm}_{G,\bar{\zeta}}(\O/\varpi^n)}(\Ind_{Z\SL(\Q_p)}^G\zeta/\varpi^n,\Pi^0/\varpi^n\Pi^0))[1/\varpi], 
		\end{align*}
		where in the bottom line, we consider the induced representation $ \Ind_{Z\SL(\Q_p)}^G\zeta \in \MOD^{\sm}_{G,\zeta}(\O)$ as in Proposition \ref{ext_Z=0}. Since $ Z\SL(\Q_p) $ is of index 4 or 8 in $ G $ (depending on whether $ p\neq 2 $ or $ p=2 $,) and the character $ \zeta $ is unitary, $ \Ind_{Z\SL(\Q_p)}^G\zeta $ is a 4-, or 8-dimensional unitary Banach space representation of $ G $. Then we can apply Proposition \ref{prop Ext=limExt} to get \[ H^i(\SL(\Q_p),\Pi)\cong  \Ext^i_{\Badm_{G,\zeta}(L)}(\Ind_{Z\SL(\Q_p)}^G\zeta\otimes_\O L,\Pi). \]
		
		Therefore, it suffices to prove that this Ext-group vanishes. We may replace $ L $ by a finite field extension so that we may assume without loss of generality, that $\Ind_{Z\SL(\Q_p)}^G\zeta\otimes_\O L $ is reducible. Since  it is finite dimensional, we then have a short exact sequence of the  form \[ 0\rightarrow \eta_1\circ \det \rightarrow \Ind_{Z\SL(\Q_p)}^G\zeta\otimes_\O L\rightarrow \Pi_1 \rightarrow 0, \]
		for some unitary character $ \eta_1:\Q_p^\times\rightarrow L^\times $ and $ \Pi_1\in\Badm_{G,\zeta}(L)$ of dimension 3, respectively 7. We then repeat this as in the proof of Proposition \ref{ext_Z=0} and see that it is enough to show that $ \Ext_{\Badm_{G,\zeta}(L)}^i(\eta\circ \det,\Pi)=0 $ for each $ i\ge 0 $ and for each unitary character $ \eta: \Q_p^\times\rightarrow L^\times$. But by the decomposition \[ \Badm_{G,\zeta}(L)\cong \bigoplus_\mathcal{B}\Badm_{G,\zeta}(L)_{\mathcal{B}}, \]
		we can reduce to the case where $ \Pi $ is a representation in the same block $ \Banb $ as $ \eta\circ\det $. More precisely, this block is given by (Corollary 1.2 in \cite{pavskunas2014blocks}) \[ \B=\{\bar{\eta}\circ\det, \Sp\otimes\bar{\eta}\circ\det, \Ind_B^G\alpha\otimes\bar{\eta}\circ\det  \}, \] where $ \bar{\eta}:\Q_p \rightarrow k$ and $ \alpha:B\rightarrow k $ are the compositions of $ \eta  $, respectively $ \tilde{\alpha} $, with the quotiont map $ \O\twoheadrightarrow k$.
		
		And by \cite[Section 6.2]{pavskunas2021finiteness}, the category $ \Banb $ decomposes as a direct sum of subcategories \[ \Banb\cong \bigoplus_{\mathfrak{n}\in\mspec Z_{\B}[1/p]}\Ban^{\adm}_{G,\zeta}(L)_{\B,\mathfrak{n}} ,\] where $ Z_{\B} $ is the center of the endomorphism ring $ \E $ which depends on the block $ \B $, defined in Section \ref{sec:banach of GL2}. The category $ \Ban^{\adm}_{G,\zeta}(L)_{\B,\mathfrak{n}} $ is defined to be the full subcategory of $ \Ban^{\adm}_{G,\zeta}(L)_{\B} $ consisting of those representations which are killed by a power of the ideal $ \mathfrak{n} $.
		Moreover, the irreducible objects of the subcategory $ \Ban^{\adm}_{G,\zeta}(L)_{\B,\mathfrak{n}} $, which contains $ \eta\circ\det $, are described in Corollary 6.10 of \cite{pavskunas2021finiteness} and are precisely those representations that we excluded. In particular, there are no extensions between $ \Pi $ and $ \eta\circ\det $.

	\end{proof}

	\bibliographystyle{siam}
	\bibliography{CtsCohAndExt-Bib}
\end{document}